\documentclass{article}
    
    \usepackage{amssymb,amsmath}
    \usepackage{amsfonts}
    \usepackage{empheq}
    \usepackage{amsthm}
    \usepackage{cleveref}
    \usepackage{graphicx,subfig}
    \usepackage[usenames]{color}
    \usepackage{enumitem}
    \usepackage{float}
    
    \title{Heteroclinic Connection in a Nicholson's delayed model with Harvesting term.}    ﻿
    \author{Adrian Gomez \thanks{e-mail:agomez@ubiobio.cl}\\{\small\itshape{Departamento de Matem\'atica,}}\\ {\small \itshape{Universidad del B\'io-B\'io, Casilla 5-C, Concepci\'on, Chile}}\and Cesar  Guayasamin \thanks{e-mail: caguayasamin@outlook.com}\\{\small\itshape{Departamento de Matem\'atica,}}\\ {\small \itshape{Universidad del B\'io-B\'io, Casilla 5-C, Concepci\'on, Chile}}}
    \date{}

    \theoremstyle{definition}
        
    \newtheorem{definition}{Definition}[section]
    \newtheorem{theorem}{Theorem}[section]
    \newtheorem{lemma}{Lemma}[section]
    
    \def\R{\mathbb{R}}

\begin{document}
    \maketitle
    
    \begin{abstract}
    ﻿
    In this paper we prove the existence of monotone heteroclinic solutions for the delayed Nicholson's blowflies model with harvesting:
    \[
    x'(t) = -\delta x(t) - Hx(t-\sigma) + \rho x(t-r)e^{-x(t-r)}.
    \]
    Under the condition $1 < \dfrac{\rho}{\delta+H} \leq e$, we establish the connection between the equilibria $0$ and $\ln(\rho/(\delta+H))$ using the Wu and Zou monotone iteration method adapted for two delays ($\sigma \neq r$). The proof combines explicit upper and lower solutions construction with characteristic equation analysis, supported by numerical simulations.
    \end{abstract}
    \bigskip
    ﻿
    {\small \noindent {\em MSC 2020 Classification} : 34K26, 34K20, 92D25, 34K12, 34K19.
    ﻿
    ﻿
    }
    ﻿
    \newpage
    ﻿
    \section{Introduction}
    
    The study of population dynamics plays a fundamental role across biological disciplines, from epidemiology to ecology and resource management. Mathematical models have emerged as essential tools for understanding and predicting complex population behaviors \cite{FBCC}. A key challenge in modeling biological systems lies in accounting for temporal delays inherent to ecological processes - including maturation periods, disease incubation, resource regeneration, and population response times \cite{TE}. Delay differential equations (DDEs) provide a powerful framework for incorporating these critical time lags, as demonstrated by their extensive applications in theoretical biology \cite{YK,Hale,LIN}.
    ﻿
    Among the most influential case studies in population dynamics is the Australian sheep blowfly (\textit{Lucilia cuprina}), a pest species whose population oscillations were first systematically documented by Nicholson in 1954 \cite{AJN}. Gurney et al. \cite{WSG} later formulated the canonical Nicholson blowfly model:
    ﻿
    \begin{equation}
    \frac{dx(t)}{dt} = \rho x(t-\tau)e^{-ax(t-\tau)} - \delta x(t)
    \end{equation}
    ﻿
    where $x(t)$ represents adult population density, $\rho$ the maximum reproduction rate, $\delta$ mortality rate, and $\tau$ the maturation delay. This equation belongs to the important class of recruitment-delay models that have shaped our understanding of delayed population dynamics \cite{LBEBI}.
    ﻿
    Recent decades have witnessed significant extensions of the Nicholson model, including (non-exhaustive list): Variable coefficient formulations \cite{LBEBI}, Distributed delay approaches, Integro-differential versions, diffusive versions, among others.
    These developments have yielded important insights into: Existence and stability of positive solutions \cite{LBEB}, Oscillation dynamics, Periodic and almost-periodic behaviors. Particularly relevant to our work, when $\frac{\rho}{\delta} \in (1,e^2]$, the system exhibits: Monotone heteroclinic connections for $r \in [0,r^*]$ (Theorem 2.3 \cite{LBEBI}), Eventually monotone, non-monotone (Doing $c=+\infty$ in Theorem 1 \cite{CHKKNT}), Slowly oscillatory heteroclinic solutions (Corollary 2.4 \cite{LBEBI})
    \begin{figure}[h]
    \centering
    \includegraphics[width=0.9\textwidth]{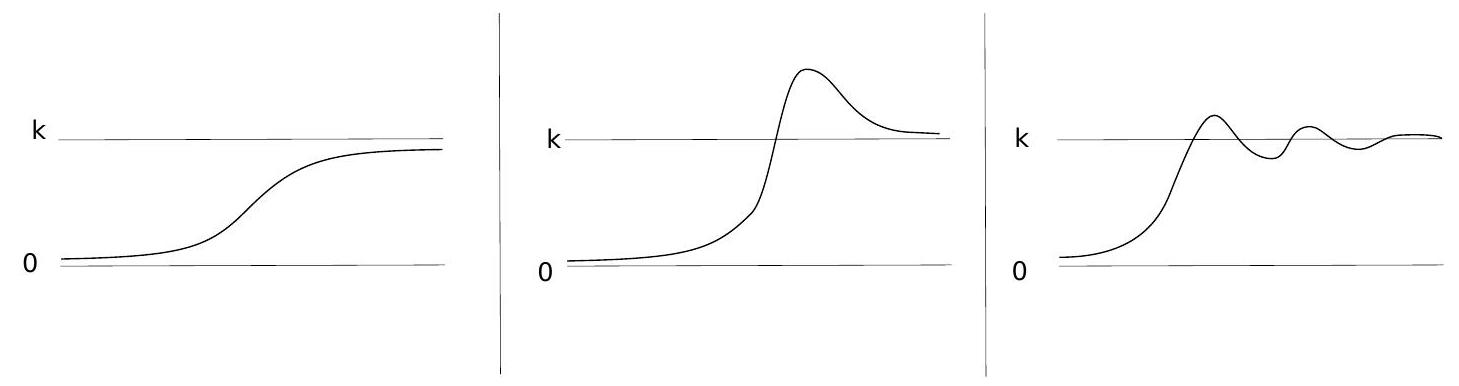}
    \caption{Heteroclinic solution types: monotone (left), eventually monotone (center), and slowly oscillatory (right)}
    \end{figure}
    ﻿
    The current study, which aims to prove the existence of heteroclinic connections for a Nicholson's model with harvesting,  builds upon the foundational framework established by Wu and Zou \cite{WZ}, which was originally designed to analyze traveling wave solutions in delayed reaction-diffusion systems of the form:
    ﻿
    \begin{equation}
    \frac{\partial u(x,t)}{\partial t} = D\frac{\partial^2 u(x,t)}{\partial x^2} + f(u_t(x)).
    \end{equation}
    ﻿
    Their approach combines three key elements: the systematic construction of ordered upper-lower solution pairs, a convergent monotone iteration scheme based on integral operators, and the critical exponential quasi-monotonicity condition, which enables the treatment of non-monotone reaction terms through nonstandard ordering techniques.
    ﻿
    While Wu and Zou's theory was originally developed for proving the existence of traveling wave solutions, in this work, we adapt  their methodology to a different setting: the study of heteroclinic orbits in a Nicholson-type equation with harvesting and two delays. Specifically, we focus on
    ﻿
    \begin{equation}\label{NichH}
    \frac{d x(t)}{d t}=\rho x(t-r) e^{-x(t-r)}-\delta x(t)-H x(t-\sigma),
    \end{equation}
    ﻿
    where $H$ denotes the harvesting rate, $\sigma$ represents the harvesting delay, and $\rho, \tau, \delta$ are the same as in equation (1). \\
    The delay $\sigma$ can be interpreted as follows:  The parameter $\sigma$ represents the remaining lifetime that an individual, removed from the population at the present time, would have had in the absence of intervention. From a demographic perspective, this can be interpreted as the additional time the individual would have remained in the population according to the survival function. In the context of ecological or fisheries models, this quantity reflects the anticipated impact of harvesting on future dynamics: by removing an individual at the present time, not only is the current population size reduced, but also its future trajectory is altered, since that individual would have contributed to the population for an additional period of length $\sigma$ had it not been harvested.
    \\
     Due to the introduction of the additional delay $\sigma > 0$, the model \eqref{NichH} becomes a system whose global dynamics remain largely unexplored. In particular, the case with delays $r \neq \sigma$ corresponds to one of the open problems proposed in \cite{LBEBI}.
    ﻿
    ﻿
    A key challenge in adapting Wu and Zou’s method to equation \eqref{NichH} is the presence of two distinct delays, which significantly complicates the construction of upper and lower solutions. To the best of our knowledge, the only prior attempt to extend this approach to equations with two delays appears in \cite{LHZX,CHXD} and in the  recent preprint \cite{BarkerMinh}, where they develop a upper and  lower solution framework for a diffusive Nicholson model. However, their models  differs substantially from the one developed here. Our contribution provides a complementary perspective by addressing the existence of heteroclinic orbits through a novel adaptation of Wu and Zou’s methodology.\\
    It is not difficult to see that if \( \frac{\rho}{\delta+H}>1 \), then the Nicholson model with a linear harvesting term (4) has two equilibria: The null equilibrium \( x_{0} \equiv 0 \) and a positive equilibrium \( x_{\kappa} \equiv \ln \left(\frac{\rho}{\delta+H}\right) \). 
    The main result obtained in this work is the following:
    \newpage
    \begin{theorem}\label{uno}
     Let  \( 1<\frac{\rho}{\delta+H} \leq e \), and  \( \lambda>0 \)  the positive root of the equation
    ﻿
    \begin{equation}\label{eq6}
    -z-\delta-H e^{-\sigma z}+\rho e^{-r z}=0. 
    \end{equation}
    ﻿
    If  the delays \( r, \sigma \) satisfy the conditions:
    ﻿
    \begin{enumerate}
      \item \( \sigma<r \);
      \item \( \sigma \in\left(0, \sigma_{0}\right] \), where \( \sigma_{0} \) satisfies the relation \( H \sigma_{0} e^{1+\sigma_{0} \delta}=1 \);
      \item \( H e^{-(\lambda+\epsilon) \sigma}-\rho e^{-(\lambda+\epsilon) r} \geq 0 \), for some \( 0<\epsilon<\lambda \);
      \end{enumerate}
    Then the Nicholson equation with delay and linear harvesting term \eqref{NichH} has a monotone heteroclinic solution connecting the equilibria \( x_{0}=0 \) and \( x_{\kappa}=\ln \left(\frac{\rho}{\delta+H}\right)>0 \).
    ﻿
    ﻿
    \end{theorem}
    ﻿
     The paper is organized as follows: In the section 1 we introduces the theoretical framework adapted from \cite{WZ} to our model. In the section 2 we apply this theory to  the Nicholson equation \eqref{NichH} and build the  upper and lower solutions. Finally we  presents numerical experiments based on the analytical approximation to the desired solution. 
    ﻿
    ﻿
    \section{Theoretical framework}
    To make this article self-contained, we adapt the theory in \cite{WZ} to functional first order equations having the form 
    ﻿
    \begin{equation}\label{edofuncional}
    -u^{\prime}(t)+f\left(u_{t}\right)=0, 
    \end{equation}
    ﻿
    for $f: C([-\tau, 0], \mathbb{R}) \rightarrow \mathbb{R}$ a continuous functional and $u_{t} \in C([-\tau, 0] ; \mathbb{R})$  given by
    ﻿
    $$
    u_{t}(s)=u(s+t), \quad s \in[-\tau, 0], t \geq 0,
    $$
    ﻿
    Note that our case in \eqref{NichH} is a particular case of \eqref{edofuncional}, with $\tau=\max\{\sigma,r\}$ and $f(\phi)=\rho \phi(-r)e^{-\phi(-r)}-\delta \phi(0)-H\phi(-\sigma)$.
    ﻿
    \begin{definition} If $u_{0} \equiv 0$ and $u_{\kappa} \equiv K$ are two equilibrium solutions of \eqref{edofuncional}, $\phi \in C^{1}(\mathbb{R}, \mathbb{R})$ is monotone, satisfies equation \eqref{edofuncional} and additionally verifies the asymptotic conditions
    ﻿
    $$
    \lim _{s \rightarrow-\infty} \phi(s)=0, \lim _{s \rightarrow+\infty} \phi(s)=K
    $$
    ﻿
    Then $\phi$ is said to be a monotone heteroclinic solution connecting $u_{0}$ and $u_{\kappa}$.
    \end{definition}
    Furthermore, given $w \in \mathbb{R}^{n}$, we define $\widehat{w} \in C([-\tau, 0] ; \mathbb{R})$ as $\widehat{w}(s)=w$ for all $s \in$ $[-\tau, 0]$.\\
    Now, we assume the following hypotheses about the functional $f$:
    \begin{description}
    \item[(H1)] There exists $K>0$ such that $f(\widehat{0})=f(\widehat{K})=0$ and $f(\widehat{u}) \neq 0$ for arbitrary $u \in(0, K)$,\\
    \item[(H2)] (Exponential quasi-monotonicity). There exists $\beta \geq 0$ such that
    ﻿
    $$
    f(\phi)-f(\psi)+\beta[\phi(0)-\psi(0)] \geq 0
    $$
    ﻿
    for all $\phi, \psi \in C([-\tau, 0], \mathbb{R})$ such that
    \begin{enumerate}
    \item[(i)] $0 \leq \psi(s) \leq \phi(s) \leq K, s \in[-\tau, 0]$,
    \item[(ii)] $e^{\beta s}(\phi(s)-\psi(s))$ is non-decreasing for $s \in[-\tau, 0]$.
    \end{enumerate}
    \end{description}
    We also define the following set $\Gamma$, called the profile set,\\
    $$\Gamma=\left\{\phi \in C(\mathbb{R}, \mathbb{R}) \left\lvert\, \begin{array}{c}\phi(-\infty)=0, \phi(+\infty)=K, \\ \phi \text{ is non-decreasing in } \mathbb{R}, \\ e^{\beta t}(\phi(t+s)-\phi(t)) \text{ is non-decreasing in } t \in \mathbb{R}, \text{ for all } s>0,\end{array}\right.\right\}.$$
    ﻿
    ﻿
    In this case, the constant $\beta$ is the same as in hypothesis (H2). Additionally, we define the operator $\mathbf{H}: C(\mathbb{R}, \mathbb{R}) \rightarrow C(\mathbb{R}, \mathbb{R})$ by
    ﻿
    \begin{equation}\label{Hdef} 
    \mathbf{H}(\phi)(t):=f\left(\phi_{t}\right)+\beta \phi_{t}(0), \quad \phi \in C(\mathbb{R}, \mathbb{R}), t \in \mathbb{R} 
    \end{equation}
    ﻿
    The following result provides some properties about the monotonicity of $\mathbf{H}$ on the profile set $\Gamma$.
    ﻿
    \begin{lemma} \label{lemma:monotonicity_H} 
    Suppose (H1) and (H2) are satisfied. Then for any $\phi \in \Gamma$, we have that
    \begin{enumerate}
        \item[(i)] $\boldsymbol{H}(\phi)(t) \geq 0$ for all $t \in \mathbb{R}$,
        \item[(ii)] $\boldsymbol{H}(\phi)(t)$ is non-decreasing in $t \in \mathbb{R}$,
        \item[(iii)] $\boldsymbol{H}(\psi)(t) \leq \boldsymbol{H}(\phi)(t)$ for all $t \in \mathbb{R}$, if $\psi \in C(\mathbb{R}, \mathbb{R})$ satisfies $0 \leq \psi(t) \leq \phi(t) \leq K$ and $e^{\beta t}[\phi(t)-\psi(t)]$ is non-decreasing in $t \in \mathbb{R}$.
    \end{enumerate}
    \end{lemma}
    ﻿
    \begin{proof}
     The proof of (i) follows directly from (H1) and (H2) by taking $\psi \equiv 0$. To prove (ii), let $t \in \mathbb{R}$ and $s>0$. Then
    ﻿
    $$
    0 \leq \phi_{t}(z) \leq \phi_{t+s}(z) \leq K, \quad z \in[-\tau, 0]
    $$
    ﻿
    and
    ﻿
    $$
    e^{\beta z}\left[\phi_{t+s}(z)-\phi_{t}(z)\right]=e^{-\beta t} e^{\beta(t+z)}[\phi(t+z+s)-\phi(t+z)]
    $$
    ﻿
    is non-decreasing in $z \in[-\tau, 0]$, since $\phi \in \Gamma$. It follows from (H2) that
    ﻿
    $$
    \begin{aligned}
    \mathbf{H}(\phi)(t+s)-\mathbf{H}(\phi)(t) & =f\left(\phi_{t+s}\right)-f\left(\phi_{t}\right)+\beta[\phi(t+s)-\phi(t)] \\
    & =f\left(\phi_{t+s}\right)-f\left(\phi_{t}\right)+\beta\left[\phi_{t+s}(0)-\phi_{t}(0)\right] \\
    & \geq 0
    \end{aligned}
    $$
    ﻿
    Thus, $\mathbf{H}(\phi)(t)$ is non-decreasing in $t \in \mathbb{R}$. Finally, for (iii). Observe that
    ﻿
    $$
    0 \leq \psi_{t}(z) \leq \phi_{t}(z) \leq K, \quad z \in[-\tau, 0]
    $$
    ﻿
    and
    ﻿
    $$
    e^{\beta z}\left[\phi_{t}(z)-\psi_{t}(z)\right]=e^{-\beta t} e^{\beta(t+z)}[\phi(t+z)-\psi(t+z)]
    $$
    ﻿
    is non-decreasing in $z \in[-\tau, 0]$. Then, by (H2)
    ﻿
    $$
    \begin{aligned}
    \mathbf{H}(\phi)(t)-\mathbf{H}(\psi)(t) & =f\left(\phi_{t}\right)-f\left(\psi_{t}\right)+\beta[\phi(t)-\psi(t)] \\
    & =f\left(\phi_{t}\right)-f\left(\psi_{t}\right)+\beta\left[\phi_{t}(0)-\psi_{t}(0)\right] \\
    & \geq 0 .
    \end{aligned}
    $$
    ﻿
    This completes the proof.
    \end{proof}
    ﻿
    ﻿
    ﻿
    In order to initiate our  iterative scheme, a upper solution of \eqref{edofuncional} is needed, which is defined as follows:
    ﻿
    \begin{definition} A upper solution of \eqref{edofuncional} is a continuous function $\varphi: \mathbb{R} \rightarrow \mathbb{R}$ whose first derivative $\varphi^{\prime}$ exists almost everywhere, is essentially bounded on $\mathbb{R}$, and satisfies the inequality
    ﻿
    $$
    -\varphi^{\prime}(t)+f\left(\varphi_{t}\right) \leq 0, \quad \text{a.e. } t \in \mathbb{R}
    $$
    ﻿
    A lower solution of \eqref{edofuncional} is a function defined similarly  with the reversed inequality.
    \end{definition}
    ﻿
    In the following results, in addition to assuming that $(H1)$ and $(H2)$ are satisfied, we also suppose that equation \eqref{edofuncional} has a upper solution $\bar{\varphi} \in \Gamma$ and a lower solution $\underline{\varphi}$ (not necessarily in $\Gamma$) satisfying:
    \begin{description}
    \item[(C1)] $0 \leq \underline{\varphi}(t) \leq \bar{\varphi}(t) \leq K, t \in \mathbb{R}$,
    \item[(C2)] $\underline{\varphi} \not \equiv 0$.
    \item[(C3)] $e^{\beta t}[\bar{\varphi}(t)-\underline{\varphi}(t)]$ is non-decreasing for $t \in \mathbb{R}$.
    \end{description}
    The first iteration involves the following linear non-homogeneous ordinary differential equation
    ﻿
    \begin{equation}\label{edoit}
    x_{1}^{\prime}(t)=-\beta x_{1}(t)+\mathbf{H}(\bar{\varphi})(t), \quad t \in \mathbb{R} 
    \end{equation}
    ﻿
    Among all solutions of equation \eqref{edoit}, we choose an appropriate solution. The definition and some properties of this particular solution are shown below.
    \newpage
    \begin{lemma} \label{lemaop} Define $x_{1}: \mathbb{R} \rightarrow \mathbb{R}$ as
    ﻿
    \begin{equation}\label{operador}
    x_{1}(t)=\int_{-\infty}^{t} e^{-\beta(t-s)} \boldsymbol{H}(\bar{\varphi})(s) d s, \quad t \in \mathbb{R} 
    \end{equation}
    ﻿
    Then
    \begin{enumerate}
    \item[(i)] $x_{1}$ is a solution of \eqref{edoit}.
    \item[(ii)] $x_{1} \in \Gamma$.
    \item[(iii)] $\underline{\varphi}(t) \leq x_{1}(t) \leq \bar{\varphi}(t)$ for $t \in \mathbb{R}$.
    \end{enumerate}
    ﻿
    \end{lemma}
    ﻿
    \begin{proof} The proof of (i) is immediate, since from \eqref{operador} we have
    ﻿
    $$
    e^{\beta t} x_{1}(t)=\int_{-\infty}^{t} e^{\beta s} \mathbf{H}(\bar{\varphi})(s) d s, \quad t \in \mathbb{R}
    $$
    ﻿
    Then,
    ﻿
    $$
    e^{\beta t} x_{1}^{\prime}(t)+\beta e^{\beta t} x_{1}(t)=e^{\beta t} \mathbf{H}(\bar{\varphi})(t), \quad t \in \mathbb{R} .
    $$
    ﻿
    This proves that $x_{1}$ satisfies equation \eqref{edoit}. For (ii), we first prove that $x_{1}$ satisfies the two limit conditions, that is
    ﻿
    $$
    \lim _{t \rightarrow-\infty} x_{1}(t)=0 \text{ and } \lim _{t \rightarrow+\infty} x_{1}(t)=K .
    $$
    ﻿
    Indeed, applying l'Hôpital's rule to \eqref{operador}, we obtain
    ﻿
    $$
    \begin{aligned}
    \lim _{t \rightarrow \infty} x_{1}(t) & =\lim _{t \rightarrow \infty} \frac{1}{e^{\beta t}} \int_{-\infty}^{t} e^{\beta s} \mathbf{H}(\bar{\varphi})(s) d s \\
    & =\lim _{t \rightarrow \infty} \frac{\mathbf{H}(\bar{\varphi})(t)}{\beta} .
    \end{aligned}
    $$
    ﻿
    Since $\mathbf{H}(\bar{\varphi})(t)=f\left(\bar{\varphi}_{t}\right)+\beta \bar{\varphi}(t)$, using $(H1)$ and the fact that $\bar{\varphi} \in \Gamma$, we get
    ﻿
    $$
    \lim _{t \rightarrow-\infty} \mathbf{H}(\bar{\varphi})(t)=f(\widehat{0})=0 \text{ and } \lim _{t \rightarrow+\infty} \mathbf{H}(\bar{\varphi})(t)=f(\widehat{K})+\beta K=\beta K .
    $$
    ﻿
    Consequently, $\lim\limits_{t \rightarrow-\infty} x_{1}(t)=0 \text{ and } \lim\limits_{t \rightarrow+\infty} x_{1}(t)=K$.\\
    Now we prove that $x_{1}$ is non-decreasing in $\mathbb{R}$. Let $t \in \mathbb{R}$ and $s>0$ be given. Then
    ﻿
    $$
    \begin{aligned}
    x_{1}(t+s)-x_{1}(t) & =\int_{-\infty}^{t+s} e^{-\beta(t+s-\theta)} \mathbf{H}(\bar{\varphi})(\theta) d \theta-\int_{-\infty}^{t} e^{-\beta(t-\theta)} \mathbf{H}(\bar{\varphi})(\theta) d \theta \\
    & =\int_{-\infty}^{t} e^{-\beta(t-\theta)} \mathbf{H}(\bar{\varphi})(\theta+s) d \theta-\int_{-\infty}^{t} e^{-\beta(t-\theta)} \mathbf{H}(\bar{\varphi})(\theta) d \theta \\
    & =\int_{-\infty}^{t} e^{-\beta(t-\theta)}[\mathbf{H}(\bar{\varphi})(\theta+s)-\mathbf{H}(\bar{\varphi})(\theta)] d \theta
    \end{aligned}
    $$
    ﻿
    By \cref{lemma:monotonicity_H} (ii), we have that $\mathbf{H}(\bar{\varphi})(\theta+s)-\mathbf{H}(\bar{\varphi})(\theta) \geq 0$ and therefore $x_{1}(t+$ $s)-x_{1}(t) \geq 0$. It remains to prove that given $s>0$, $e^{\beta t}\left[x_{1}(t+s)-x_{1}(t)\right]$ is non-decreasing in $t \in \mathbb{R}$. Indeed,
    ﻿
    $$
    \begin{aligned}
    e^{\beta t}\left[x_{1}(t+s)-x_{1}(t)\right] & =e^{\beta t} \int_{-\infty}^{t} e^{-\beta(t-\theta)}[\mathbf{H}(\bar{\varphi})(\theta+s)-\mathbf{H}(\bar{\varphi})(\theta)] d \theta \\
    & =\int_{-\infty}^{t} e^{\beta \theta}[\mathbf{H}(\bar{\varphi})(\theta+s)-\mathbf{H}(\bar{\varphi})(\theta)] d \theta
    \end{aligned}
    $$
    ﻿
    Then, again by  \cref{lemma:monotonicity_H} (ii) we obtain that
    ﻿
    $$
    \frac{d}{d t}\left(e^{\beta t}\left[x_{1}(t+s)-x_{1}(t)\right]\right)=e^{\beta t}[\mathbf{H}(\bar{\varphi})(t+s)-\mathbf{H}(\bar{\varphi})(t)] \geq 0
    $$
    ﻿
    showing  (ii).\\
    To prove (iii), let $w(t)=\bar{\varphi}(t)-x_{1}(t), t \in \mathbb{R}$. It follows from \eqref{Hdef} and the fact that $\bar{\varphi}$ is a upper solution of \eqref{edofuncional} that $\bar{\varphi}^{\prime}(t) \geq-\beta \bar{\varphi}(t)+\mathbf{H}(\bar{\varphi})(t), t \in \mathbb{R}$. Combining this with \eqref{edoit}, we obtain
    ﻿
    $$
    w^{\prime}(t)+\beta w(t) \geq 0, \quad t \in \mathbb{R}
    $$
    ﻿
    Let $r(t)=w^{\prime}(t)+\beta w(t), t \in \mathbb{R}$ and  note that $r$ is essentially bounded. From the fundamental theory of first-order linear ordinary differential equations, we get
    ﻿
    $$
    w(t)=c e^{-\beta t}+\int_{-\infty}^{t} e^{-\beta(t-s)} r(s) d s, \quad t \in \mathbb{R}
    $$
    ﻿
    where $c$ is a constant. Now note that $$\lim\limits_{t \rightarrow-\infty} w(t)=\lim\limits _{t \rightarrow-\infty} x_{1}(t)-\lim\limits _{t \rightarrow-\infty} \bar{\varphi}(t)=0.$$
    ﻿
    ﻿
     Therefore, we must have $c=0$. Thus, since $r(t) \geq 0, t \in \mathbb{R}$,
    ﻿
    \begin{equation}\label{diferencia}
    w(t)=\bar{\varphi}(t)-x_{1}(t)=\int_{-\infty}^{t} e^{-\beta(t-s)} r(s) d s \geq 0, \quad t \in \mathbb{R} 
    \end{equation}
    ﻿
    This proves that $x_{1}(t) \leq \bar{\varphi}(t), t \in \mathbb{R}$.\\
    Similarly defining $v(t)=x_{1}(t)-\underline{\varphi}(t), t \in \mathbb{R}$ we have, by \cref{lemma:monotonicity_H} (iii) and hypotheses (C1) and (C3), that
    ﻿
    $$
    \begin{aligned}
    v^{\prime}(t) & =x_{1}^{\prime}(t)-\underline{\varphi}^{\prime}(t) \\
    & =-\beta x_{1}(t)+\mathbf{H}(\bar{\varphi})(t)+\beta \underline{\varphi}(t)-\mathbf{H}(\underline{\varphi})(t) \\
    & =-\beta v(t)+[\mathbf{H}(\bar{\varphi})(t)-\mathbf{H}(\underline{\varphi})(t)] \\
    & \geq-\beta v(t), \quad t \in \mathbb{R} .
    \end{aligned}
    $$
    ﻿
    Let $h(t)=v^{\prime}(t)+\beta v(t), t \in \mathbb{R}$. Then $h(t) \geq 0, t \in \mathbb{R}$ and $h$ is essentially bounded. Thus,
    ﻿
    $$
    v(t)=c e^{-\beta t}+\int_{-\infty}^{t} e^{-\beta(t-s)} h(s) d s, \quad t \in \mathbb{R}
    $$
    ﻿Note that $\lim_{t\rightarrow 0} v(t)=0$, since $0\leq \underline{\varphi}(t) \leq \overline{\varphi}(t)$, $ t\in \mathbb{R}$. This fact and the boundedness of $h$ allows us to deduce that $c=0$ must hold. Then,
    ﻿
    \begin{equation}\label{difsub}
    v(t)=x_{1}(t)-\underline{\varphi}(t)=\int_{-\infty}^{t} e^{-\beta(t-s)} h(s) d s \geq 0, \quad t \in \mathbb{R} 
    \end{equation}
    ﻿
    This completes the proof.
    ﻿
    \end{proof}
    ﻿
    \begin{lemma}\label{lemmaSS} Let $x_{1}$ be as in \cref{lemaop}. Then
    \begin{enumerate}
    \item[(i)] $e^{\beta t}\left[\bar{\varphi}(t)-x_{1}(t)\right]$ is non-decreasing in $t \in \mathbb{R}$,
    \item[(ii)] $e^{\beta t}\left[x_{1}(t)-\underline{\varphi}(t)\right]$ is non-decreasing in $t \in \mathbb{R}$,
    \item[(iii)] $x_{1}$ is a upper solution of \eqref{edofuncional}.
    \end{enumerate}
    \end{lemma}
    \begin{proof} 
    To prove (i), we use the equality \eqref{diferencia} from the proof of \cref{lemaop} to obtain
    ﻿
    $$
    e^{\beta t}\left[\bar{\varphi}(t)-x_{1}(t)\right]=e^{\beta t} \int_{-\infty}^{t} e^{-\beta(t-s)} r(s) d s=\int_{-\infty}^{t} e^{\beta s} r(s) d s, \quad t \in \mathbb{R},
    $$
    ﻿
    from which
    ﻿
    $$
    \frac{d}{d t}\left(e^{\beta t}\left[\bar{\varphi}(t)-x_{1}(t)\right]\right)=e^{\beta t} r(t) \geq 0, \quad t \in \mathbb{R}.
    $$
    ﻿
    Its proves that $e^{\beta t}\left[\bar{\varphi}(t)-x_{1}(t)\right]$ is non-decreasing in $t \in \mathbb{R}$. The proof of (ii) is completely analogous, using equality \eqref{difsub}.\\
    Finally, to prove (iii), from equation \eqref{edoit} and     \cref{lemma:monotonicity_H} (iii), we obtain
    ﻿
    $$
    \begin{aligned}
    x_{1}^{\prime}(t) & =-\beta x_{1}(t)+\mathbf{H}(\bar{\varphi})(t) \\
    & =f\left(x_{1 t}\right)+\left[\mathbf{H}(\bar{\varphi})(t)-\mathbf{H}\left(x_{1}\right)(t)\right] \\
    & \geq f\left(x_{1 t}\right), \quad t \in \mathbb{R},
    \end{aligned}
    $$
    ﻿
    that is, $-x_{1}^{\prime}(t)+f\left(x_{1 t}\right) \leq 0$. This completes the proof.
    \end{proof}
    ﻿
    Since $x_{1}$ is also a upper solution, we can replace $\bar{\varphi}$ by $x_{1}$ and consider the following linear non-homogeneous ordinary differential equation:
    ﻿
    $$
    x_{2}^{\prime}(t)=-\beta x_{2}(t)+\mathbf{H}\left(x_{1}\right)(t), \quad t \in \mathbb{R} .
    $$
    ﻿
    In this way we can repeat the previous procedure. In general, we consider the following iterative scheme:
    ﻿
    ﻿
    ﻿
    ﻿
    ﻿
    \begin{empheq}[left=\empheqlbrace]{align}
    x_{m}^{\prime}(t)&=-\beta x_{m}(t)+\mathbf{H}\left(x_{m-1}\right)(t), \qquad t \in \mathbb{R}, m=1,2, \ldots  \label{esquema}\\[0.2cm]
    x_{0}(t)&=\bar{\varphi}(t), \qquad t \in \mathbb{R}.\nonumber
    \end{empheq}
    ﻿
    ﻿
    ﻿
    By solving \eqref{esquema} inductively and repeating the same argument for $x_{1}$, we obtain a sequence of functions $\left\{x_{m}\right\}_{m=1}^{\infty}$ with the following properties:
    \begin{description}
    \item[(P1)] $x_{m} \in \Gamma, m=1,2, \ldots$
    \item[(P2)] $0 \leq \underline{\varphi}(t) \leq x_{m}(t) \leq x_{m-1}(x) \leq \bar{\varphi}(t), \in \mathbb{R}, m=1,2, \ldots$,
    \item[(P3)] $e^{\beta t}\left[x_{m-1}(t)-x_{m}(t)\right]$ is non-decreasing in $\mathbb{R}$,
    \item[(P4)] $e^{\beta t}\left[x_{m}(t)-\underline{\varphi}(t)\right]$ is non-decreasing in $\mathbb{R}$.
    \end{description}
    From property (P2) it follows that $\lim\limits_{m \rightarrow \infty} x_{m}(t)=x(t), t \in \mathbb{R}$, exists and $\underline{\varphi}(t) \leq x(t) \leq$ $\bar{\varphi}(t)$. Moreover, $x$ is non-decreasing, because each $x_{m}$ is. Next we prove that $x$ is a solution of \eqref{edofuncional}.
    ﻿
    \begin{lemma} \label{lemalim} Let $x(t)=\lim\limits_{m \rightarrow \infty} x_{m}(t)$, hence $x(t)$ is a solution of \eqref{edofuncional} and it satisfies
    ﻿
    $$
    \lim _{t \to-\infty} x(t)=0 \text{ and } \lim _{t \to+\infty} x(t)=K.
    $$
    \end{lemma}
    ﻿
    ﻿
    \begin{proof} It follows from the continuity of $f$ and Lebesgue's Dominated Convergence Theorem that
    ﻿
    \begin{align}
    x(t)=\lim _{m \rightarrow \infty} x_{m}(t) & =\lim _{m \rightarrow \infty} \int_{-\infty}^{t} e^{-\beta(t-s)} \mathbf{H}\left(x_{m-1}\right)(s) d s\nonumber \\
    & =\int_{-\infty}^{t} e^{\beta(t-s)} \lim _{m \rightarrow \infty} \mathbf{H}(x_{m-1})(s) d s\label{limite}  \\
    & =\int_{-\infty}^{t} e^{-\beta(t-s)} \mathbf{H}(x)(s) d s\nonumber
    \end{align}
    ﻿
    Then,
    ﻿
    $$
    \begin{aligned}
    x^{\prime}(t) & =-\beta \int_{-\infty}^{t} e^{-\beta(t-s)} \mathbf{H}(x)(s) d s+e^{-\beta t} e^{\beta t} \mathbf{H}(x)(t) \\
    & =-\beta x(t)+\mathbf{H}(x)(t), \quad t \in \mathbb{R} .
    \end{aligned}
    $$
    ﻿
    Therefore,
    ﻿
    $$
    x^{\prime}(t)=-\beta x(t)+f\left(x_{t}\right)+\beta x(t), \quad t \in \mathbb{R},
    $$
    ﻿
    that is
    ﻿
    $$
    -x^{\prime}(t)+f\left(x_{t}\right)=0, \quad t \in \mathbb{R}.
    $$
    ﻿
    Thus, $x$ is a solution of \eqref{edofuncional}.\\
    Now from (P2) and the fact that $\lim\limits_{t \to-\infty} \bar{\varphi}(t)=0$, we obtain that $\lim\limits_{t \to-\infty} x(t)=0$.  On the other hand, $x$ is non-decreasing and bounded above by $K$. Therefore $\lim\limits _{t \to+\infty} x(t)=K_{0}$ exists and
    ﻿
    $$
    \sup _{t \in \mathbb{R}} \underline{\varphi}(t) \leq K_{0} \leq K
    $$
    ﻿
    It follows from hypotheses (C1) and (C2) that $K_{0} \in(0, K]$. Using equality \eqref{limite}, it follows from l'Hôpital's rule and the continuity of $f$ that
    ﻿
    $$
    \begin{aligned}
    K_{0}=\lim _{t \rightarrow+\infty} x(t) & =\lim _{t \rightarrow+\infty} \frac{1}{e^{\beta t}} \int_{-\infty}^{t} e^{\beta s} \mathbf{H}(x)(s) d s \\
    & =\lim _{t \rightarrow+\infty} \frac{\mathbf{H}(x)(t)}{\beta} \\
    & =\frac{f\left(\widehat{K_{0}}\right)+\beta K_{0}}{\beta} \\
    & =\frac{f\left(\widehat{K_{0}}\right)}{\beta}+K_{0},
    \end{aligned}
    $$
    ﻿
    from which, $f\left(\widehat{K_{0}}\right)=0$. However, by $(\mathrm{H}1)$ we conclude that $K_{0}=K$. This completes the proof.
    \end{proof}
    Lemmas \ref{lemaop}, \ref{lemmaSS} and \ref{lemalim} are summarized in the following result.\\
    \begin{theorem}\label{principal} Suppose (H1) and (H2) are satisfied and that equation \eqref{edofuncional} has a upper solution $\bar{\varphi} \in \Gamma$ and a lower solution $\underline{\varphi}$ satisfying (C1)-(C3). Then \eqref{edofuncional} has a monotone solution satisfying $u(-\infty)=0$ and $u(+\infty)=K$.
    \end{theorem}
    \section{Application to the Nicholson Model}
    ﻿
    In this section, we prove the existence of a heteroclinic connection for the Nicholson model with  harvesting term given in \eqref{NichH}. We start checking the quasi-monotonicity condition (H2) and after, we construct the upper and lower solutions, in order to apply the method presented in the last section. 
    In what follows, we assume:
    \begin{description}
    \item[(A1)] $1<\frac{p}{\delta+H} \leq e$;
    \item[(A2)] $0<\sigma<r$.
    \end{description}
    ﻿
    ﻿
    \subsection*{3.1. Quasi-Monotonicity Condition}
    we claim that the functional $f$ given by
    \begin{equation}\label{fun}
    f(\phi)=-\delta \phi(0)-H\phi(-\sigma)+\rho \phi(-r)e^{-\phi(-r)},\quad \phi\in C([-r,0],\R),
    \end{equation} 
    satisfies the exponential quasi-monotonicity condition (H2). Indeed it is show in the following 
    ﻿
    \begin{lemma}\label{lemaH2} Assuming  (A1), for any $\beta>\delta+H$ and $\sigma$  sufficiently small, the functional $f$ satisfies condition (H2).
    \end{lemma}
    ﻿
    \begin{proof} Let $\phi, \psi \in C([-\tau, 0], \mathbb{R})$ such that $0 \leq \psi(s) \leq \phi(s) \leq \kappa$ and $e^{\beta s}(\phi(s)-\psi(s))$ is non-decreasing for all $s \in[-\tau, 0]$. Then,
    ﻿
    \begin{equation}\label{ineq}
    \phi(-\sigma)-\psi(-\sigma) \leq e^{\beta \sigma}(\phi(0)-\psi(0)), 
    \end{equation}
    ﻿
    and
    ﻿
    \begin{align}
    f(\phi)-f(\psi)= & -\delta(\phi(0)-\psi(0))-H(\phi(-\sigma)-\psi(-\sigma)) \nonumber \\
    & +\rho\left(\phi(-r) e^{-\phi(-r)}-\psi(-r) e^{-\psi(-r)}\right) \label{qmon}
    \end{align}
    ﻿
    Since $1<\frac{\rho}{\delta+H} \leq e$, then
    ﻿
    $$
    0 \leq \psi(s) \leq \phi(s) \leq 1
    $$
    ﻿
    Moreover, since the function $h(y)=y e^{-y}, y \in \mathbb{R}$ is increasing on $[0,1]$, it follows that:
    ﻿
    $$
    \phi(-r) e^{-\phi(-r)}-\psi(-r) e^{-\psi(-r)}>0
    $$
    ﻿
    Therefore, from \eqref{qmon} we have:
    ﻿
    $$
    f(\phi)-f(\psi) \geq-\delta(\phi(0)-\psi(0))-H e^{\beta \sigma}(\phi(0)-\psi(0)) .
    $$
    ﻿
    It follows from \eqref{ineq} and the above inequality that:
    ﻿
    $$
    f(\phi)-f(\psi)+\beta(\phi(0)-\psi(0)) \geq\left(\beta-\delta-H e^{\beta \sigma}\right)(\phi(0)-\psi(0)),
    $$
    ﻿
    which implies that (H2) is satisfied whenever $e^{\beta \sigma} \leq(\beta-\delta) / H$, which will hold for any $\beta>\delta+H$ when $\sigma$ is sufficiently small.
    \end{proof}
    ﻿
    \subsection*{3.2. Existence of Upper and Lower Solutions}
    Now, we propose a pair of upper and lower solutions for equation \eqref{NichH}. Using elementary calculus techniques, we can prove the following technical lemmas:
    ﻿
    \begin{lemma}\label{lema_realroot}
        For equation \eqref{NichH}, the quasi-polynomial characteristic equation around the zero equilibrium:
    ﻿
    \begin{equation*}\label{char0}
    \chi_{0}(z):=-z-\delta-H e^{-\sigma z}+\rho e^{-r z} 
    \end{equation*}
    ﻿
    has a real root $\lambda>0$.
    ﻿
    \end{lemma}
    %
    %
    %
    %
    %
    ﻿
    \begin{lemma}\label{lema_cond1}
     Let $\sigma_{0}>0$ satisfy the relation:
    ﻿
    \begin{equation}\label{condsigma0}
    1=\sigma_{0} H e^{\left(1+\sigma_{0} \delta\right)} 
    \end{equation}
    ﻿
    Then for any $\sigma \in\left(0, \sigma_{0}\right]$, there exists $\epsilon>0$ such that for $\mu \in\left[\mu_{0}-\epsilon, \mu_{0}+\epsilon\right]$, where $\mu_{0}=\frac{1}{\sigma}+\delta$, the inequality holds:
    ﻿
    \begin{equation}\label{mu0}
    \mu-\delta-H e^{\sigma \mu} \geq 0 
    \end{equation}
    \end{lemma}
    ﻿
    ﻿
    \begin{proof}
    To determine the value of $\sigma_{0}$, consider the point $\mu_{0}$ where the derivative of the function $\mu \rightarrow H^{\sigma_{0} \mu}, \mu \in \mathbb{R}$ coincides with the slope of the line $\mu-\delta$. Then, $\mu_{0}$ must satisfy the equations:
    ﻿
    $$
    \begin{aligned}
    \mu_{0}-\delta & =H e^{\sigma_{0} \mu_{0}} \\
    1 & =\sigma_{0} H e^{\sigma_{0} \mu_{0}}
    \end{aligned}
    $$
    ﻿
    from which $\mu_{0}=\frac{1}{\sigma_{0}}+\delta$ and therefore $\sigma_{0}$ is such that:
    ﻿
    $$
    1=\sigma_{0} H e^{1+\sigma_{0} \delta}
    $$
    ﻿
    Now, consider the function $\mu \rightarrow \mu-\delta-H e^{\sigma \mu}=: g(\mu)$, for $\mu \in \mathbb{R}$. If $\sigma \in\left(0, \sigma_{0}\right]$, take $\mu_{0}:=\frac{1}{\sigma}+\delta$, then:
    ﻿
    $$
    g\left(\mu_{0}\right)=\frac{1}{\sigma}-H e^{\sigma\left(\frac{1}{\sigma}-\delta\right)}>\frac{1}{\sigma}-H e^{1+\sigma_{o} \delta}=\frac{1}{\sigma}-\frac{1}{\sigma_{0}} \geq 0 .
    $$
    ﻿
    By the continuity of the function $g$, there exists a neighborhood of $\mu_{0},\left(\mu_{0}-\epsilon, \mu_{0}+\epsilon\right)$ such that $\mu-\delta-H e^{\sigma \mu} \geq 0$.
    \end{proof}
    ﻿
    Now, in the following lemmas, we construct the upper and lower solutions for \eqref{NichH} and check the conditions (C1), (C2) and (C3). 
    \begin{lemma}\label{supersol} Let $\sigma \in\left(0, \sigma_{0}\right]$ and $\mu$ be as in \cref{lema_cond1}. The function $\bar{\varphi}: \mathbb{R} \rightarrow \mathbb{R}$ defined by:
    ﻿
    \[
    \bar{\varphi}(t):= \begin{cases}\dfrac{\mu}{\mu+\lambda} \kappa e^{\lambda t}, & t \leq 0,  \label{dsuper}\\ & \\ \kappa\left(1-\dfrac{\lambda}{\mu+\lambda} e^{-\mu t}\right), & t>0.\end{cases}
    \]
    ﻿
    is a upper solution for \eqref{NichH}.
    \end{lemma}
    ﻿
    \begin{proof}
     Substituting \eqref{dsuper} into \eqref{NichH}, we consider 4 cases:
     \begin{description}
    \item[Case 1:] If $t \in(-\infty, 0]$:
    ﻿
    $$
    \begin{aligned}
    & -\bar{\varphi}^{\prime}(t)-\delta \bar{\varphi}(t)-H \bar{\varphi}(t-\sigma)+\rho \bar{\varphi}(t-r) e^{-\bar{\varphi}(t-r)} \\
    \leq & -\bar{\varphi}^{\prime}(t)-\delta \bar{\varphi}(t)-H \bar{\varphi}(t-\sigma)+\rho \bar{\varphi}(t-r) \\
    = & -\frac{\lambda \mu}{\lambda+\mu} \kappa e^{\lambda t}-\delta \frac{\mu}{\lambda+\mu} \kappa e^{\lambda t}-H \frac{\mu}{\lambda+\mu} \kappa e^{\lambda(t-\sigma)}+\rho \frac{\mu}{\lambda+\mu} \kappa e^{\lambda(t-r)} \\
    = & \frac{\kappa \mu}{\lambda+\mu} e^{\lambda t}\left[-\lambda-\delta-H e^{-\lambda \sigma}+\rho e^{-\lambda r}\right] \\
    = & 0
    \end{aligned}
    $$
    ﻿
    \item[Case 2:] If $t \in(0, \sigma]$:
    ﻿
    $$
    \begin{aligned}
    & -\bar{\varphi}^{\prime}(t)-\delta \bar{\varphi}(t)-H \bar{\varphi}(t-\sigma)+\rho \bar{\varphi}(t-r) e^{-\bar{\varphi}(t-r)} \\
    \leq & -\frac{\lambda \mu \kappa}{\lambda+\mu} e^{-\mu t}-\delta \kappa\left(1-\frac{\lambda}{\lambda+\mu} e^{-\mu t}\right)-H \frac{\mu \kappa}{\lambda+\mu} e^{\lambda(t-\sigma)}+p \kappa e^{-\kappa} \\
    = & -\frac{\lambda \mu \kappa}{\lambda+\mu} e^{-\mu t}-\delta \kappa+p \kappa e^{-\kappa}+\delta \frac{\lambda \kappa}{\lambda+\mu} e^{-\mu t}-H \frac{\mu \kappa}{\lambda+\mu} e^{\lambda(t-\sigma)} \\
    = & -\frac{\lambda \mu \kappa}{\lambda+\mu} e^{-\mu t}+H \kappa+\delta \frac{\lambda \kappa}{\lambda+\mu} e^{-\mu t}-H \frac{\mu \kappa}{\lambda+\mu} e^{\lambda(t-\sigma)}.
    \end{aligned}
    $$
    ﻿
    For $t \in[0, \sigma]$, define $$g(t):=-\frac{\lambda \mu \kappa}{\lambda+\mu} e^{-\mu t}+H \kappa+\delta \frac{\lambda \kappa}{\lambda+\mu} e^{-\mu t}-H \frac{\mu \kappa}{\lambda+\mu} e^{\lambda(t-\sigma)}.$$
    Observe that:
    $$
    \begin{aligned}
    g(\sigma) & =-\frac{\kappa \lambda}{\lambda+\mu} e^{-\mu \sigma}(\mu-\delta)+H \kappa-H \frac{\kappa \mu}{\mu+\lambda} \\
    & =-\frac{\kappa \lambda}{\lambda+\mu} e^{-\mu \sigma}(\mu-\delta)+H \frac{\kappa \lambda}{\mu+\lambda} \\
    & =-\frac{\kappa \lambda}{\lambda+\mu} e^{-\mu \sigma}\left(\mu-\delta-H e^{\mu \sigma}\right) \\
    & \leq 0
    \end{aligned}
    $$
    ﻿
    Moreover, the derivative $g^{\prime}(t) \geq 0$ in $[0, \sigma]$. Indeed:
    ﻿
    $$
    \begin{aligned}
    g^{\prime}(t) & =\frac{\kappa \lambda \mu}{\lambda+\mu} e^{-\mu t}(\mu-\delta)-H \frac{\kappa \lambda \mu}{\lambda+\mu} e^{\lambda(t-\sigma)} \\
    & \geq \frac{\kappa \lambda \mu}{\lambda+\mu}\left(e^{-\mu t} H e^{\mu \sigma}-H e^{\lambda(t-\sigma)}\right) \\
    & =\frac{\kappa \lambda \mu H}{\lambda+\mu}\left(e^{-\mu(t-\sigma)}-e^{\lambda(t-\sigma)}\right) \\
    & \geq 0,
    \end{aligned}
    $$
    ﻿
    for all $t \in[0, \sigma]$. In conclusion, $g(t) \leq 0$ in $[0, \sigma]$, which completes this case.\\
    \item[Case 3:] If $t \in[\sigma, r]$:
    ﻿
    $$
    \begin{aligned}
    & -\bar{\varphi}^{\prime}(t)-\delta \bar{\varphi}(t)-H \bar{\varphi}(t-\sigma)+\rho \bar{\varphi}(t-r) e^{-\bar{\varphi}(t-r)} \\
    = & -\frac{\lambda \mu \kappa}{\lambda+\mu} e^{-\mu t}-\delta \kappa\left(1-\frac{\lambda}{\lambda+\mu} e^{-\mu t}\right)-H \kappa\left(1-\frac{\lambda}{\lambda+\mu} e^{-\mu(t-\sigma)}\right)+\rho \bar{\varphi}(t-r) e^{-\bar{\varphi}(t-r)} \\
    = & -\frac{\lambda \kappa}{\lambda+\mu} e^{-\mu t}\left[\mu-\delta-H e^{\mu \sigma}\right]-\kappa(\delta+H)+\rho \bar{\varphi}(t-r) e^{-\bar{\varphi}(t-r)} \\
    = & -\frac{\lambda \kappa}{\lambda+\mu} e^{-\mu t}\left[\mu-\delta-H e^{\mu \sigma}\right]-\kappa \rho e^{-\kappa}+\rho \bar{\varphi}(t-r) e^{-\bar{\varphi}(t-r)} \\
    = & -\frac{\lambda \kappa}{\lambda+\mu} e^{-\mu t}\left[\mu-\delta-H e^{\mu \sigma}\right]-\rho\left[\kappa e^{-\kappa}-\bar{\varphi}(t-r) e^{-\bar{\varphi}(t-r)}\right] \\
    \leq & 0 .
    \end{aligned}
    $$
    ﻿
    \item[Case 4:] If $t \in[r,+\infty)$, the proof is the same as in Case 3.
    \end{description}
    \end{proof}
    
    ﻿\noindent
    The proposed upper solution and its verification,  i.e. 
    $$\Phi(t):=-\bar{\varphi}^{\prime}(t)-\delta \bar{\varphi}(t)-H \bar{\varphi}(t-\sigma)-\rho \bar{\varphi}(t-r) e^{-\bar{\varphi}(t-r)} \leq 0, \ \mbox{a.e.}\ t \in \R,$$  are showing in  \cref{vsuper}.
    
    \begin{figure}[H]
    \begin{center}
    \includegraphics[width=12.5cm]{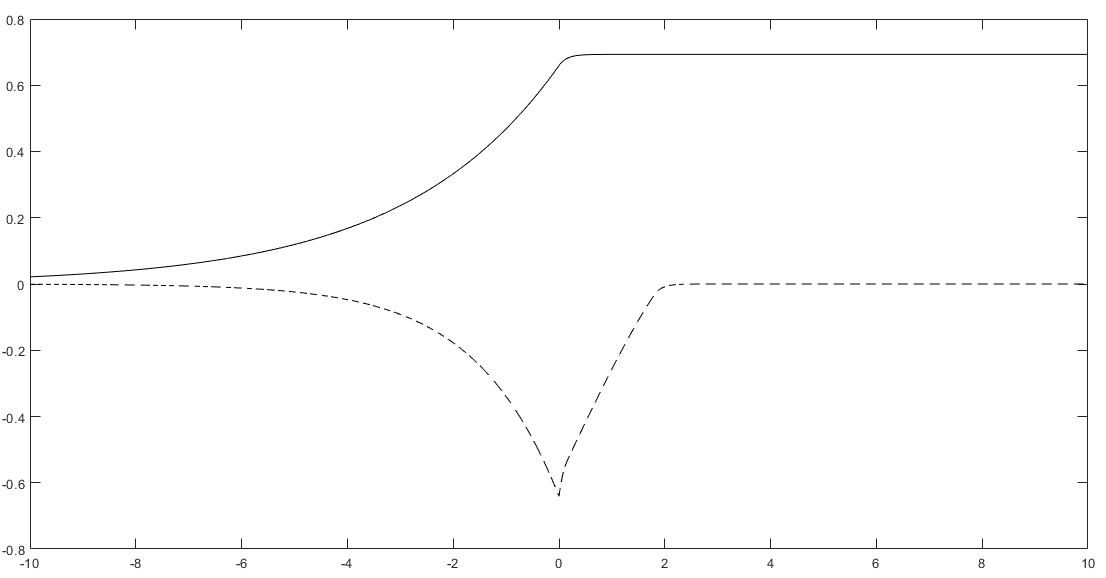}
    \caption{Graph of upper solution $\bar{\varphi}$, with $\mu=6.71, \lambda=0.3420$, coefficients $\delta=1, H=2$, $\rho=6$, delays $\sigma=0.15, r=1.8$ and its verification.}
    \label{vsuper}
    \end{center}
    \end{figure}
    ﻿
    In the following lemma, we construct the lower solution for \eqref{NichH}:
    \begin{lemma}\label{lema_subsol}  Let $t_{0}<0$ and let $0<\epsilon<\lambda$ and $\alpha>0$ satisfy:
    \begin{description}
        \item[(A)]
    ﻿
    \begin{equation}\label{difexp}
    H e^{-(\lambda+\epsilon) \sigma}-\rho e^{-(\lambda+\epsilon) r} \geq 0 
    \end{equation}
    ﻿
    \item[(B)]
    ﻿
    \begin{equation}
    \alpha<\frac{1}{\rho} \min \left\{\lambda+\delta+H e^{-(\lambda+\epsilon) \sigma}-\rho e^{-(\lambda+\epsilon) r} , \frac{-\chi_{0}(\lambda+\epsilon) (\lambda+\epsilon)^{\frac{\lambda+\epsilon}{\epsilon}} }{4  \epsilon^{2} (\lambda-\epsilon)^{\frac{\lambda- \epsilon}{\epsilon}}}\right\} .   \label{alphac}
    \end{equation}
    \end{description}
    Then, the function $\underline{\varphi}: \mathbb{R} \rightarrow \mathbb{R}$ defined by:
    ﻿
    \[
    \underline{\varphi}(t):= \begin{cases}\alpha\left(1-e^{\epsilon\left(t-t_{0}\right)}\right) e^{\lambda t}, & t \leq t_{0},  \label{dsub}\\ & \\ 0, & t>t_{0}.\end{cases}
    \]
    ﻿
    is a lower solution for \eqref{NichH}.
    \end{lemma}
    ﻿
    \begin{proof} As in the previous lemma, we divide the proof into cases:
    \begin{description}
        \item[Case 1:] $t \in\left(-\infty, t_{0}\right]$:   Using  the inequality $e^{-x} \geq 1-x$, for all $x \in \mathbb{R}$, we obtain:
    $$
    \begin{aligned}
    & -\underline{\varphi}^{\prime}(t)-\delta \underline{\varphi}(t)-H \underline{\varphi}(t-\sigma)+\rho \underline{\varphi}(t-r) e^{-\underline{\varphi}(t-r)} \\
    \geq & -\underline{\varphi}^{\prime}(t)-\delta \underline{\varphi}(t)-H \underline{\varphi}(t-\sigma)+\rho \underline{\varphi}(t-r)-\rho \underline{\varphi}^{2}(t-r) \\
    = & -\alpha\left(\lambda-(\epsilon+\lambda) e^{\epsilon\left(t-t_{0}\right)}\right) e^{\lambda t}-\delta \alpha\left(1-e^{\epsilon\left(t-t_{0}\right)}\right) e^{\lambda t}-H \alpha\left(1-e^{\epsilon\left(t-t_{0}-\sigma\right)}\right) e^{\lambda(t-\sigma)} \\
    + & \rho \alpha\left(1-e^{\epsilon\left(t-t_{0}-r\right)}\right) e^{\lambda(t-r)}-\rho \varphi^{2}(t-r) \\
    = &  \, \alpha e^{\lambda t}\left(-\lambda-\delta-H e^{-\lambda \sigma}+\rho e^{-\lambda r}\right)+ \alpha e^{(\lambda+\epsilon) t-\epsilon t_{0}}\left((\lambda+\epsilon)+\delta+H e^{-(\lambda+\epsilon) \sigma}-\rho e^{-(\lambda+\epsilon) r}\right)\\
    -&\rho \underline{\varphi}^{2}(t-r) \\
    = & -\chi_{0}(\lambda+\epsilon) \alpha e^{(\epsilon+\lambda) t-\epsilon t_{0}}-\rho \alpha^{2}\left(1-e^{\epsilon\left(t-t_{0}-r\right)}\right)^{2} e^{2 \lambda(t-r)} \\
    = & \alpha e^{(\epsilon+\lambda) t-\epsilon t_{0}}\left(-\chi_{0}(\epsilon+\lambda)-\rho \alpha\left(e^{\frac{-\epsilon\left(t-t_{0}-r\right)}{2}}-e^{\frac{\epsilon\left(t-t_{0}-r\right)}{2}}\right)^{2} e^{\lambda t-r(2 \lambda+\epsilon)}\right) \\
    = & \alpha e^{(\epsilon+\lambda) t-\epsilon t_{0}}\left(-\chi_{0}(\epsilon+\lambda)-\rho \alpha\left(e^{\frac{(\lambda-\epsilon) t+\epsilon\left(t_{0}+r\right)}{2}}-e^{\frac{(\lambda+\epsilon) t-\epsilon\left(t_{0}+r\right)}{2}}\right)^{2} e^{-r(2 \lambda+\epsilon)}\right) \\
    \geq &   \, \alpha e^{(\epsilon+\lambda) t-\epsilon t_{0}}\left(-\chi_{0}(\epsilon+\lambda)-\rho \alpha\left(e^{\frac{(\lambda-\epsilon) t+\epsilon\left(t_{0}+r\right)}{2}}-e^{\frac{(\lambda+\epsilon) t-\epsilon\left(t_{0}+r\right)}{2}}\right)^{2} e^{-r \lambda}\right) .
    \end{aligned}
    $$
    ﻿
    Let $h(t)=\left(e^{\frac{(\lambda-\epsilon) t+\epsilon\left(t_{0}+r\right)}{2}}-e^{\frac{(\lambda+\epsilon) t-\epsilon\left(t_{0}+r\right)}{2}}\right)^{2}$.  It is straightforward to see that  $\lim\limits _{t \rightarrow-\infty} h(t)=0$ and that  the global maximum value of $h(t)$ is attained at $t^{*}$, given by:
    $$
    t^{*}=\frac{1}{\epsilon} \ln \left(\frac{\lambda-\epsilon}{\lambda+\epsilon}\right)+t_{0}+r .
    $$
    Moreover,
    $$
    h\left(t^{*}\right)=\frac{4 \epsilon^{2}}{(\lambda+\epsilon)^{2}}\left(\frac{\lambda-\epsilon}{\lambda+\epsilon}\right)^{\frac{\lambda-\epsilon}{\epsilon}} e^{\lambda\left(t_{0}+r\right)}.
    $$
    Hence, note that the right side is larger than
    $$\alpha e^{(\epsilon+\lambda)t} \left(  -\chi_{0}(\epsilon+\lambda) -\rho \alpha \frac{4 \epsilon^{2}}{(\lambda+\epsilon)^{2}}\left(\frac{\lambda-\epsilon}{\lambda+\epsilon}\right)^{\frac{\lambda-\epsilon}{\epsilon}}\right) .  $$
    
    ﻿
    ﻿
    Therefore, from \eqref{difexp} we conclude this case.
    \item[Case 2:] $t \in\left[t_{0}, t_{0}+\sigma\right]$: 
    Note that $t-r \leq t_{0}$, from which
    \begin{equation*}
    	\underline{\varphi}(t-r)=\alpha \left( 1-e^{\epsilon(t-r-t_{0})}\right)e^{\lambda (t-r)} \leq \alpha e^{\lambda (t-r)}. 
    \end{equation*}
    Using this fact and the inequality $e^{-x} \geq 1-x$, for all $x \in \mathbb{R}$, we obtain:
    ﻿
    $$
    \begin{aligned}
    & -\underline{\varphi}^{\prime}(t)-\delta \underline{\varphi}(t)-H \underline{\varphi}(t-\sigma)+\rho \underline{\varphi}(t-r) e^{-\underline{\varphi}(t-r)} \\
    \geq & -\underline{\varphi}^{\prime}(t)-\delta \underline{\varphi}(t)-H \underline{\varphi}(t-\sigma)+\rho \underline{\varphi}(t-r)-\rho \underline{\varphi}^{2}(t-r) \\
    =& -\alpha H\left(1-e^{\epsilon\left(t-t_{0}-\sigma\right)}\right) e^{\lambda(t-\sigma)}+\alpha \rho\left(1-e^{\epsilon\left(t-t_{0}-r\right)}\right) e^{\lambda(t-r)}-\rho \underline{\varphi}^{2}(t-r)\\
    \geq & \, \alpha e^{\lambda t}\left(-H e^{-\lambda \sigma}+\rho e^{-\lambda r}\right)+\alpha e^{(\lambda+\epsilon) t-\epsilon t_{0}}\left(H e^{-\sigma(\lambda+\epsilon)}-\rho e^{-r(\lambda+\epsilon)}\right)-\rho \alpha^{2}  e^{2 \lambda(t-r)} \\
    \geq & \, \alpha e^{\lambda t_{0}}(\lambda+\delta)+\alpha e^{\lambda t_{0}}\left(H e^{-(\lambda+\epsilon) \sigma}-\rho e^{-(\lambda+\epsilon) r}\right)-\rho \alpha^{2}  e^{ \lambda t_{0}}   \\
    = & \, \alpha e^{\lambda t_{0}}\left(\lambda+\delta+H e^{-(\lambda+\epsilon) \sigma}-\rho e^{-(\lambda+\epsilon) r}-\alpha \rho \right)\\
    \geq & \,  0,
    \end{aligned}
    $$
    where the last inequality follows from \eqref{difexp}.
    \item[Case 3:] $t \in\left[t_{0}+\sigma,+\infty\right)$: This is trivial.
    \end{description}
    \end{proof}
    \noindent
    The graph of the proposed lower solution and its verification, i.e. $$\Psi(t):=-\underline{\varphi}^{\prime}(t)-\delta \underline{\varphi}(t)-H \underline{\varphi}(t-\sigma)-\rho \underline{\varphi}(t-r) e^{-\underline{\varphi}(t-r)} \geq 0\mbox{ a.e. } t \in \mathbb{R}$$ are showed in the following figure:
    ﻿
    \begin{figure}[H]
    \begin{center}
    \includegraphics[width=12.5cm]{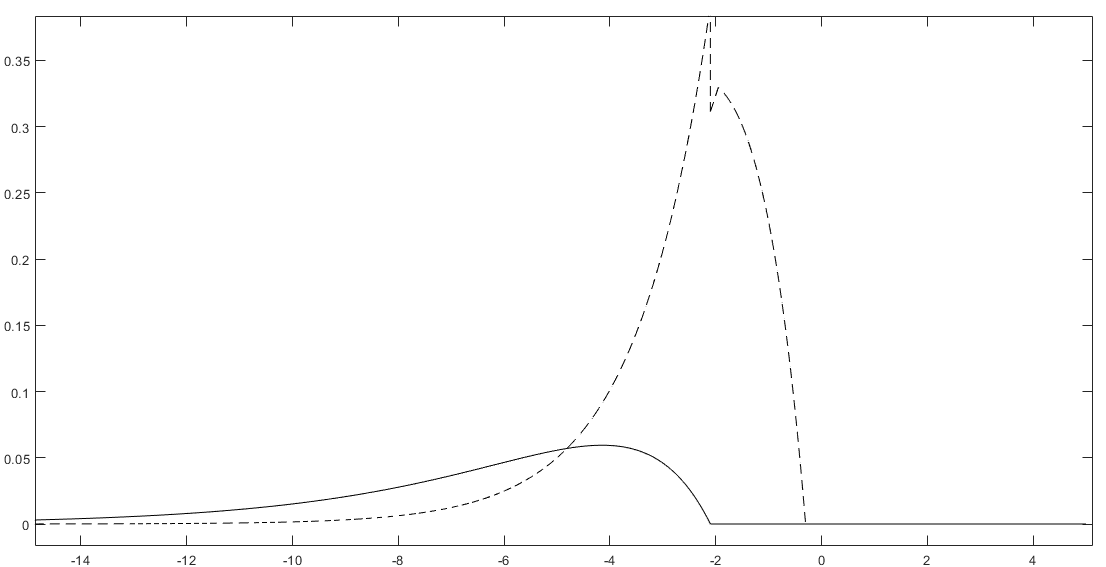}
    \caption{ Graph of lower solution $\underline{\varphi}$, with $\alpha=0.5, \epsilon=0.33$, coefficients $\delta=1, H=2$, $\rho=6$, delays $\sigma=0.15, r=1.8$ and its verification.
    }
    \label{vsub}
    \end{center}
    \end{figure}
    ﻿
    ﻿
    ﻿
    Next, we show that the proposed upper solution belongs to the profile set $\Gamma$. For this, consider $\mu=\beta$.
    \begin{lemma}\label{subperfil}
     The upper solution \eqref{dsuper} belongs to the profile set $\Gamma$.
     \end{lemma}
    \begin{proof}  First, it is easy to see that:
    ﻿
    $$
    \bar{\varphi}(-\infty)=0 \quad \text{and} \quad \bar{\varphi}(+\infty)=\kappa .
    $$
    ﻿and that $\bar{\varphi}$ is increasing. In order to  show that for all $s>0$, the function $\gamma(t):=e^{\beta t}[\bar{\varphi}(t+s)-\bar{\varphi}(t)]$, for $t \in \mathbb{R}$, is non-decreasing, we consider 3 cases.
    \begin{description}
    \item[Case 1:] If $t<0$ and $t+s \leq 0$, then
    ﻿
    $$
    \begin{aligned}
    \gamma(t) & =e^{\mu t}\left[\kappa \frac{\mu}{\mu+\lambda} e^{\lambda(t+s)}-\kappa \frac{\mu}{\mu+\lambda} e^{\lambda t}\right] \\
    & =\kappa \frac{\mu}{\mu+\lambda} e^{(\mu+\lambda) t}\left(e^{\lambda s}-1\right)
    \end{aligned}
    $$
    ﻿
    Then,
    ﻿
    $$
    \gamma^{\prime}(t)=\kappa \mu\left(e^{\lambda s}-1\right) e^{(\mu+\lambda) t}>0
    $$
    ﻿
    \item[Case 2:] If $t \leq 0$ and $t+s>0$, then
    ﻿
    $$
    \gamma(t)=e^{\mu t}\left[\kappa\left(1-\frac{\lambda}{\mu+\lambda} e^{-\mu(t+s)}\right)-\kappa \frac{\mu}{\mu+\lambda} e^{\lambda t}\right] .
    $$
    ﻿
    Then,
    ﻿
    $$
    \begin{aligned}
    \gamma^{\prime}(t) & =\kappa e^{\mu t}\left[\mu\left(1-\frac{\lambda}{\lambda+\mu} e^{-\mu(t+s)}\right)-\frac{\mu^{2}}{\mu+\lambda} e^{\lambda t}+\frac{\lambda \mu}{\mu+\lambda} e^{-\mu(t+s)}-\frac{\mu \lambda}{\lambda+\mu} e^{\lambda t}\right] \\
    & =\kappa e^{\mu t}\left[\mu-\frac{\mu}{\lambda+\mu} e^{\lambda t}(\mu+\lambda)\right]=\kappa \mu e^{\mu t}\left(1-e^{\lambda t}\right) \geq 0
    \end{aligned}
    $$
    ﻿
    \item[Case 3:] If $t>0$ and $t+s>0$, then
    ﻿
    $$
    \begin{aligned}
    \gamma(t) & =\kappa e^{\mu t}\left[-\frac{\lambda}{\mu+\lambda} e^{-\mu(t+s)}+\frac{\lambda}{\mu+\lambda} e^{-\mu t}\right] \\
    & =\kappa \frac{\lambda}{\mu+\lambda}\left(1-e^{-\mu s}\right),
    \end{aligned}
    $$
    ﻿
    from which $\gamma^{\prime}(t)=0$.
    \end{description}
      \end{proof}
    ﻿
    The following lemma  proves that the upper solution \eqref{dsuper} and the lower solution \eqref{dsub} are compatible, i.e., they satisfy conditions (C1) and (C3).
    \begin{lemma}\label{lemasubsuper}
    Let $\alpha>0$ be as in \cref{lema_subsol} and additionally suppose that $\alpha \leq \dfrac{\kappa \mu}{\lambda+\mu}$. Then the upper solution $\bar{\varphi}$ and lower solution $\underline{\varphi}$ given in \eqref{dsuper} and \eqref{dsub}, respectively, satisfy:
    ﻿
    \begin{enumerate}
      \item $\underline{\varphi}(t) \leq \bar{\varphi}(t)$, for all $t \in \mathbb{R}$;
      \item The function $t \rightarrow e^{\beta t}[\bar{\varphi}(t)-\underline{\varphi}(t)]$, for $t \in \mathbb{R}$ is non-decreasing.
    \end{enumerate}
    \end{lemma}
    ﻿
    \begin{proof} First, observe that if $t>t_{0}$, then clearly $\underline{\varphi}(t) \leq \bar{\varphi}(t)$. Now, if $t \leq t_{0}<0$, then $0<1-e^{\epsilon\left(t-t_{0}\right)} \leq 1$. Thus,
    ﻿
    $$
    \underline{\varphi}(t)=\alpha\left(1-e^{\epsilon\left(t-t_{0}\right)}\right) e^{\lambda t} \leq \alpha e^{\lambda t} \leq \kappa \frac{\mu}{\lambda+\mu} e^{\lambda t}=\bar{\varphi}(t)
    $$
    ﻿
    Finally, observe that if $t>t_{0}$, then $e^{\beta t}[\bar{\varphi}(t)-\underline{\varphi}(t)]=e^{\beta t} \bar{\varphi}(t)$, which is clearly non-decreasing on $\left(t_{0},+\infty\right)$ by the monotonicity of the upper solution $\bar{\varphi}$. If $t<t_{0}$, then:
    ﻿
    $$
    e^{\beta t}[\bar{\varphi}(t)-\underline{\varphi}(t)]=e^{\lambda t}\left(\frac{\kappa \mu}{\lambda+\mu}-\alpha+e^{\epsilon\left(t-t_{0}\right)}\right),
    $$
    ﻿
    whose derivative is:
    ﻿
    $$
    e^{\lambda t}\left(\lambda\left(\frac{\kappa \mu}{\lambda+\mu}-\alpha\right)+(\epsilon+1) e^{\epsilon\left(t-t_{0}\right)}\right) \geq 0
    $$
    ﻿
    thus we have the desired monotonicity on $\left(-\infty, t_{0}\right]$.
    \end{proof}
    ﻿
    ﻿
    Figure \ref{supersub} compares the upper solution $\bar{\varphi}$ and lower solution $\varphi$ to verify their compatibility. The dashed curve is the graph of the function $t \rightarrow e^{\beta t}(\bar{\varphi}(t)-\underline{\varphi}(t))$, for $t \in \mathbb{R}$, and it is observed to be non-decreasing. We consider the same data as in Figure \ref{vsuper} and Figure \ref{vsub}  for the upper solution and lower solution, respectively.
    \begin{figure}[H]
    \begin{center}
    \includegraphics[width=12.5cm]{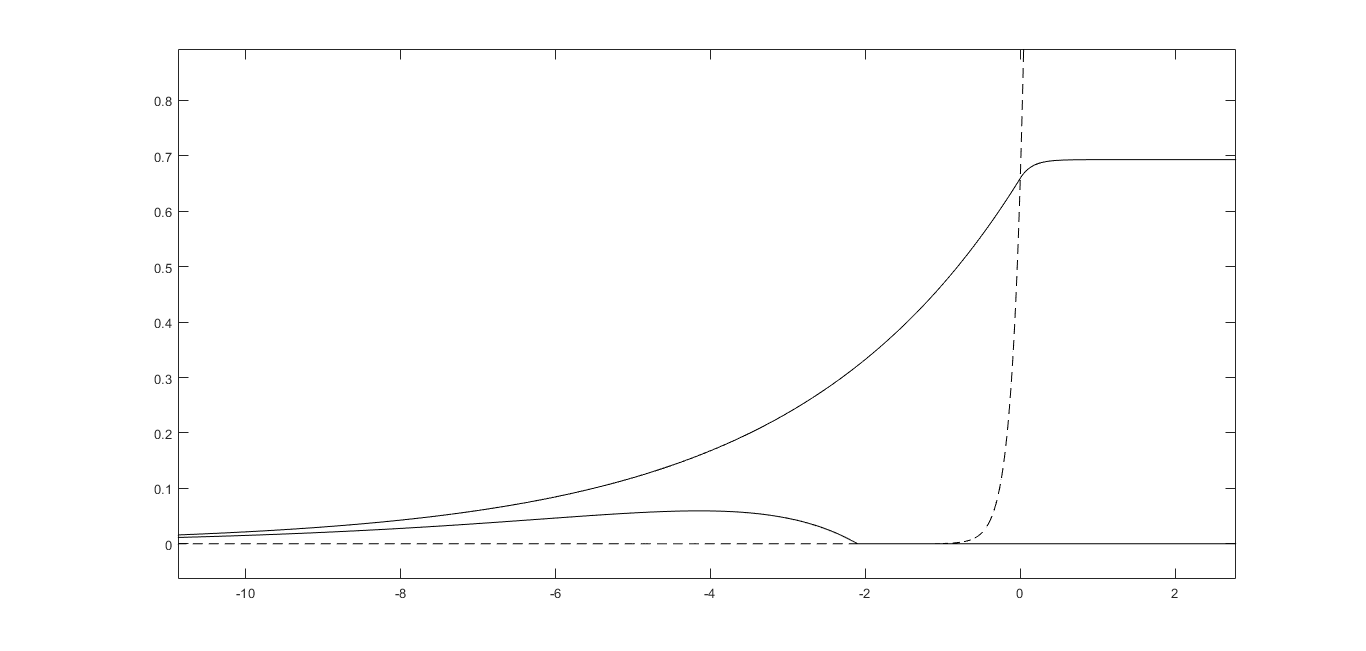}
    \caption{Compatibility of upper solution and lower solution}
    \label{supersub}
    \end{center}
    \end{figure}
    ﻿
\subsection*{Proof of Theorem \ref{uno}}       In summary, assuming (A1), (A2), (A3), and by the Lemmas \ref{supersol}, \ref{lema_subsol}, \ref{subperfil}, and \ref{lemasubsuper},  we have established the existence of a pair of upper and lower solutions for the Nicholson's blowflies model involving a linear harvesting term with delay \eqref{NichH}, satisfies conditions (C1)-(C3), and with the associated functional $f$ satisfies hypotheses (H1) and (H2). Therefore, Theorem \ref{principal} implies the main Theorem \ref{uno}.

    \subsection*{3.3. Numerical Approximation}
    The iterative scheme \eqref{esquema}  allows us to obtain a sequence of functions
     $\left\{x_{m}\right\}_{m=1}^{\infty}$ such that, for $t \in \mathbb{R}$
    ﻿
    $$
    \lim _{m \rightarrow \infty} x_{m}(t)=: x(t) \text{ exists } \quad \text{and} \quad \underline{\varphi}(t) \leq x(t) \leq \bar{\varphi}(t)
    $$
    ﻿
    Moreover, $x$ is a non-decreasing  solution of \eqref{edofuncional} and satisfies the asymptotic conditions
    ﻿
    $$
    \lim _{t \rightarrow-\infty} x(t)=0 \text{ and } \lim _{t \rightarrow+\infty} x(t)=K.
    $$
    ﻿
    The functions $x_{m}$ for $m=0, \ldots$, are given by
    ﻿
    \begin{empheq}[left=\empheqlbrace]{align}
    x_{m}(t)=&\int_{-\infty}^{t} e^{-\beta(t-s)} \mathbf{H}\left(x_{m-1}\right)(s) d s,   \label{iterativo}\\[0.2cm]
    x_{0}(t)=&\bar{\varphi}(t), \quad t \in \mathbb{R}.\nonumber
    \end{empheq}
    
    ﻿
    ﻿
    Applied to the Nicholson model with linear harvesting term
    ﻿
    $$
    x^{\prime}(t)=-\delta x(t)-H x(t-\sigma)+\rho x(t-r) e^{-x(t-r)}
    $$
    ﻿
    the operator $\mathbf{H}: C(\mathbb{R} ; \mathbb{R}) \rightarrow C(\mathbb{R} ; \mathbb{R})$ becomes
    ﻿
    $$
    \begin{aligned}
    \mathbf{H}(\phi)(s): & =\beta \phi_{s}(0)+f\left(\phi_{s}\right) \\
    & =(\beta-\delta) \phi(s)-H \phi(s-\sigma)+\rho \phi(s-r) e^{-\phi(s-r)}, 
    \end{aligned}
    $$
for   \( \phi \in C(\mathbb{R}, \mathbb{R}), s \in \mathbb{R}\). To plot the approximations given by scheme \eqref{iterativo}, we consider the specific values $\delta=1, H=2, \rho=6, \sigma=0.15, r=1.8$. The iteration is initiated with the upper solution \eqref{dsuper}, that is
    \[
    \bar{\varphi}(t):= \begin{cases}
    \kappa \dfrac{\mu}{\mu+\lambda} e^{\lambda t}, & t \leq 0, \\
    & \\
    \kappa\left(1-\dfrac{\lambda}{\mu+\lambda} \, e^{-\mu t}\right), & t > 0.
    \end{cases}
    \]
    ﻿
    where $\kappa=0.6931 \approx \ln \left(\frac{\rho}{\delta+H}\right), \mu=\beta=6.7093$ is such that
    ﻿
    $$
    \mu-\delta-H e^{\sigma \mu}>0,
    $$
    ﻿
    and $\lambda \approx 0.3420$ satisfies the characteristic equation at equilibrium $x_{0}=0$,
    ﻿
    $$
    \chi_{0}(\lambda)=-\lambda-\delta-H e^{-\lambda \sigma}+\rho e^{-\lambda r}=0.
    $$
    ﻿
    Below we plot the first 4 iterations using the integral formula in \eqref{iterativo} (Figure \ref{iteraciones}).
    \begin{figure}[H]
    \begin{center}
    \includegraphics[ width=12.5cm]{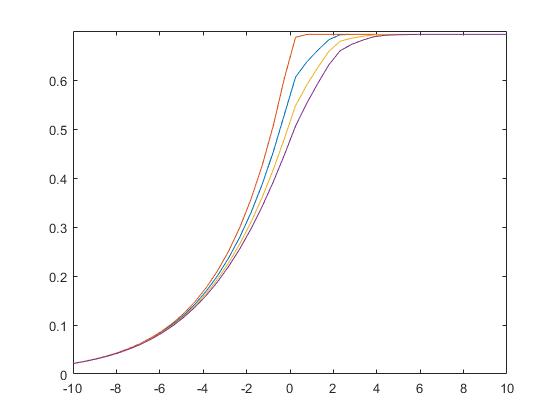}
    \caption{First 4 iterations approximating the monotone heteroclinic solution connecting the equilibria 0 and $\kappa$.}
    \label{iteraciones}
    \end{center}
    \end{figure}
    ﻿\section*{Conclusions}

In this paper we demonstrate the existence of monotone heteroclinic solutions for the Nicholson model with two distinct delays ($\sigma\neq r$) and linear harvesting term, addressing a scenario that has so far not been treated with monotone iteration techniques in the literature. By connecting the equilibria $0$ and $\ln\left(\frac{\rho}{\delta+H}\right)$ under the conditions $1 < \frac{\rho}{\delta+H} \leq e$ and $\sigma < r$, we partially answer the open problem proposed in \cite{LBEBI} on global dynamics with multiple delays for the Nicholson model, via an adaptation of Wu and Zou's method in \cite{WZ} that included the explicit construction of new upper- and lower-solutions with their respective handling of the analytical complexities of having two different delays. Finally, the numerical simulations performed validate our results.
\section*{Funding}
This research was partially supported by Universidad del Bío-Bío (Internal Regular Project RE2448106 to A.G.)

    \end{document}